\date{}
\title{Unimodular random one-ended planar graphs are sofic}

\author{\'Ad\'am Tim\'ar}

\documentclass[10pt,naturalnames]{article}
\renewcommand\footnotemark{}
\usepackage[ruled,vlined]{algorithm2e}
\usepackage{amsmath}
\usepackage{amsthm}
\usepackage{amsfonts}
\usepackage{graphicx}
\newif\ifhyper\IfFileExists{hyperref.sty}{\hypertrue}{\hyperfalse}
\ifhyper\usepackage{hyperref}\fi
\usepackage{enumitem}
\usepackage{subfig}
\usepackage{caption}
\usepackage{keyval}
\usepackage[margin=1in]{geometry}
\usepackage{setspace}
\usepackage{dsfont}
\usepackage[displaymath, mathlines,pagewise]{lineno}

\theoremstyle{definition}

\newtheorem{theorem}{Theorem}
\newtheorem{corollary}[theorem]{Corollary}
\newtheorem{lemma}[theorem]{Lemma}

\newtheorem{remark}[theorem]{Remark}
\newtheorem{proposition}[theorem]{Proposition}

\newtheorem{question}[theorem]{Question}

\newcommand{\R}{\mathbb{R}}

\def \P {{\Bbb P}}

\def \_reg {\rightarrow_{\bf reg}}

\def\maxdeg/{\Delta}

\def \P {{\bf P}}

\def \_reg {\rightarrow_{\bf reg}}

\def\maxdeg/{\Delta}

\def\G{{\cal G}}

\def\calG{{\cal G}}

\begin{document}
\maketitle
\let\thefootnote\relax\footnotetext{\footnotesize{Partially supported by the ERC Consolidator Grant 772466 ``NOISE''.}}

\begin{abstract}
We prove that if a unimodular random graph is almost surely planar and has finite expected degree, then it has a combinatorial embedding into the plane which is also unimodular. This implies the claim in the title immediately 
by a theorem of Angel, Hutchcroft, Nachmias and Ray \cite{AHNR}.
Our unimodular embedding also implies that all the dichotomy results of \cite{AHNR} about unimodular maps extend in the one-ended case to unimodular random planar graphs.
\end{abstract}

\section{Introduction, definitions}
\subsection{Results and motivation}
By a {\it unimodular random planar graph (URPG)} we mean a unimodular random graph that is almost surely planar. Such a graph is called sofic if it has a local weak approximation by a sequence of finite graphs. See Subsection \ref{definitions} for the precise definitions.  
We will prove the following.
\begin{theorem}\label{sofic}
Every unimodular random one-ended planar graph $G$ is sofic.
\end{theorem}

A subgraph $H$ of a rooted graph $(G,o)$ is a unimodular subgraph if their joint distribution ($G$ with indicator marks for the subgraph) is unimodular. This also implies that $(H_o,o)$ is unimodular, where $H_o$ is the component of $o$ in $H$. The local weak limit of convergent unimodular random graphs is unimodular, and by restricting $G$ to its subgraph induced by vertices of degree at most $k$, we get a sequence $G_k$ that converges to $G$, so it is enough to prove the theorem for graphs of bounded degree.
Unimodular trees are sofic, as proved by Elek \cite{E1} relying on a method by Bowen \cite{B}. An alternative proof was given in \cite{BLS}. This implies, as shown by Elek and Lippner in \cite{EL}, that the existence of a unimodular spanning tree is sufficient for soficity. Hence Theorem \ref{sofic} follows from the next theorem. 

\begin{theorem}\label{main}
Every unimodular random one-ended planar graph $G$ of finite expected degree contains a 
unimodular spanning tree. 
\end{theorem}

Given a graph $H$, an {\it end} of $H$ is an equivalence class of infinite non self-intersecting paths, where two such paths are equivalent if there is a third one that intersects both of them infinitely many times. 
A unimodular graph has either 0 (the finite case), 1, 2 or infinitely many ends, \cite{AL}.

Our main contribution is the following theorem. We provide the definitions after the theorem.

\begin{theorem}\label{embedding}
Let $G$ be a unimodular random planar graph of finite expected degree. Then $G$ has a unimodular combinatorial embedding into the plane. 
\end{theorem}

A {\it planar map} is defined as a proper embedding of a locally finite planar graph $G$ into an open subset $U$ of the sphere, up to orientation-preserving homeomorphisms, and with the property that every compact set of $U$ is intersected by finitely many embedded edges, and every face (component of $U$ minus the embedded graph) is homeomorphic to an open disk. We call a planar map simply connected if the union of the closure of the faces of the map is homeomorphic to either the entire sphere, or the sphere minus one point. If (after applying the homeomorphism, if necessary) one projects this embedding on the plane stereographically from this exceptional point (or an arbitrary point outside of the embedded graph, if there is no exceptional point), one gets a planar embedding with the property that every compact set is intersected by finitely many edges. We will also apply the term {\it simply connected} to an embedding of a graph, if the planar map defined by it is simply connected.

There is a combinatorial definition of embeddings: a {\it combinatorial embedding} is a collection of cyclic permutations $\pi_v$ on the set of edges incident to $v$, as $v\in V(G)$, and we call this combinatorial embedding planar if there is an embedding of $G$ into the sphere where the clockwise cyclic order of the edges around $v$ is $\pi_v$. It is clear that every embedding defines a combinatorial embedding (if we take the cyclic permutation of edges by reading them clockwise around the embedded vertex), and conversely, any combinatorial embedding can be generated by some actual embedding, by definition.
Using these permutations, one can give a combinatorial definition of {\it faces}: walk along edges, and when walking along $e$ and reaching endpoint $v$, continue along $\pi_v(e)$. See e.g. \cite{N} for a more precise definition. If a combinatorial embedding is generated by a planar map, the faces of the map are in natural bijection with the combinatorial faces. Conversely, a planar combinatorial embedding defines a unique planar map, which can be obtained if for every combinatorial 
face we take a disk, and glue its boundary along the face.
Given a combinatorial embedding, defined by permutations $\pi_v$ as $v\in V(G)$, we call this combinatorial embedding unimodular if the $\pi_v$ as markings on the vertices define a unimodular marked graph, and if this holds, we also call the resulting rooted planar map unimodular.

In \cite{AHNR} Angel, Hutchcroft, Nachmias and Ray prove that every simply connected unimodular random rooted planar map is sofic. 
A unimodular random map that represents a one-ended graph is automatically simply connected (Proposition \ref{simplyconnected}), hence Theorem \ref{oneended} applies to it. 

In this paper the only type of embeddings that we consider are combinatorial embeddings, and specifically, unimodular ones. About actual embeddings that are unimodular or invariant with regard to the automorphisms of $G$ or the underlying space, we refer the reader to joint works of the author with Benjamini \cite{BT} and with T\'oth \cite{TT}.

The next theorem is essentially Theorem 5.13 and Theorem 2 in \cite{AHNR}. 

\begin{theorem}{(Angel, Hutchcroft, Nachmias, Ray \cite{AHNR})}\label{oneended}
If $G$ has finite expected degree and it can be represented by a simply connected unimodular random rooted planar map, then the Free Uniform Spanning Forest (FUSF) of $G$ is a unimodular spanning tree almost surely. Consequently, $G$ is sofic.
\end{theorem}


Conley, Gaboriau, Marks and Tucker-Drob, \cite{CGMT}, have proved results related to Theorem \ref{oneended}. In our setup these can be vaguely phrased as follows: a unimodular simply connected planar map has {\it some} unimodular spanning tree, under certain mild conditions. Their paper is in the context of Borel graphs, and they consider graphs with a Borel 2-basis for the cycle space, which is equivalent to the existence of a Borel embedding that defines a simply connected planar map by a result of Thomassen \cite{Th}, see \cite{CGMT}. They prove that if such an embedding exists and the graph is locally finite and not two-ended, then it has a Borel spanning forest with only one-ended trees. This defines a one-ended spanning tree for the dual graph. In \cite{CGMT} it is also shown that Theorem \ref{main} (and hence Theorem \ref{sofic}) holds for the case of planar {\it Cayley graphs}.

Suppose that a URPG $G$ has a unimodular embedding into the plane. Then Theorem \ref{oneended} provides us with a sufficient condition to be sofic: whenever this embedding is simply connected.
The graph structure of $G$ does not directly determine whether the embedding is simply connected. To see this, consider the Cayley graph corresponding to $<a,b \, |\, a^3, b^3>$, which is the free product of two copies of the 3-element cyclic group. (Recall that if two groups $G_1=<S_1\,|\, R_1>$ and $G_2=<S_2\,|\, R_2>$ are given in terms of defining relators, 
where $S_i$ is a generatings set for $G_i$ and $R_i$ are defining relators, then the free product of $G_1$ and $G_2$ is $<S_1, S_2\,|\, R_1, R_2>$.)
This graph does have unimodular embeddings that are simply connected and that are not (Figure \ref{triangular} hints the proof of this claim, with the explanation below it). As the next observation shows, for one-ended URPG's, any unimodular embedding is simply connected. This is close to a characterization, because ``most" URPG's with 2 or infinitely many ends can only be embedded in the sphere with more than 1 accumulation points of edges, and hence there is no simply connected planar map that would represent them. See Remark \ref{TT} for more on the existence of simply connected embeddings. 

\begin{figure}[h]
\vspace{0.1in}
\begin{center}
\includegraphics[keepaspectratio,scale=1.5]{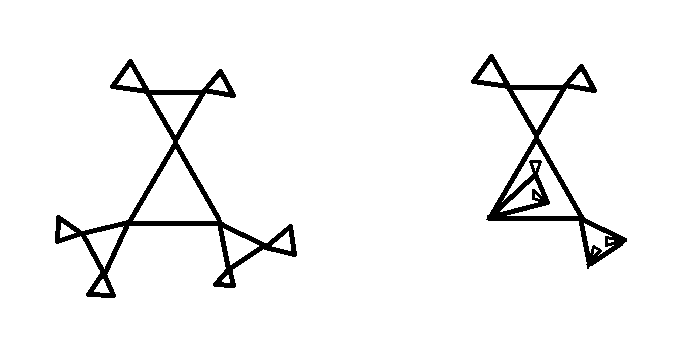}
\caption{A representation of (part of the) same infinite graph by a simply connected map (on the left) and by one that is not. The embedding on the left is unimodular (being an atomic probability measure on a single decorated graph), and we can make the one on the right unimodular by deciding for each triangle $\Delta$ indepedently how its three neighboring triangles should be mapped with regard to the two 
components of the sphere minus 
$\Delta$. After making these decisions for every triangle, the combinatorial embedding is determined up to orientation, and one can decide about the latter by a coin flip.}\label{triangular}
\end{center}
\end{figure}

\begin{proposition}\label{simplyconnected}
Let $G$ be a one-ended URPG that has a unimodular combinatorial embedding into the plane. Then this embedding defines a simply connected planar map. 
\end{proposition}
\begin{proof}
The planar map defined by a combinatorial embedding can only be a surface that is compact or one-ended (as a topological space). Otherwise one could find finitely many faces whose removal cuts the surface into at least two pieces that each contain infinitely many faces, and hence $G$ would have at least two ends. The only compact planar surface is the sphere, and the only one-ended planar surface is the sphere minus one point, as explained in Section 7 of \cite{AHNR} based on \cite{BR}. 
\qed
\medskip
\end{proof}
Hence, if $G$ also has finite expected degree, the conclusion of Theorem \ref{oneended} holds. To illustrate that the one-endedness assumption cannot be omitted in general, note that the Cartesian product of a 3-cycle with a biinfinite path or the Cartesian product of one edge and a 3-regular tree are planar, but they cannot be represented by simply connected planar maps.

\begin{remark}\label{TT}
In \cite{TT} a graph theoretical characterization is given for URPG's that have some simply connected embedding. Then it is shown, based on the method of Section \ref{sec_embedding}, that for such graphs there is also simply connected unimodular combinatorial embedding. Theorems \ref{main} and \ref{sofic} remain true for these URPG's as well, by the same argument as here.
\end{remark}

Theorem \ref{embedding} has some further corollaries. Without loss of generality we will assume that our URPG is ergodic (extremal), and will skip saying ``almost surely".

\begin{corollary}
The dichotomy results of Theorem 1 in \cite{AHNR} are valid for every unimodular random one-ended planar graph $G$. In particular, the following are equivalent. 
\begin{itemize}
\item $G$ is invariantly amenable;
\item there is a unimodular planar map with average curvature 0 that represents $G$;
\item every harmonic Dirichlet function is a constant;
\item Bernoulli($p$) percolation has at most one infinite component for every $p\in [0,1]$ almost surely.
\end{itemize}
\end{corollary}


\begin{proofof}{Theorem \ref{main}}
It follows from Theorem \ref{oneended}, Proposition \ref{simplyconnected} and Theorem \ref{embedding}.
\qed
\medskip
\end{proofof}

\begin{proofof}{Theorem \ref{sofic}}
Follows from Theorem \ref{main} essentially by \cite{EL}. See Section 8 of \cite{AHNR} for an elaboration of this implication and further references. A brief sketch of the proof is the following. Every unimodular random tree is {\it strongly sofic}, i.e., for any unimodular decoration of the tree there exists a sofic approximation by finite decorated graphs. Once we have a unimodular spanning tree $T$ in a unimodular graph $G$, we can encode $G$ as a unimodular decoration of $T$, and from the finite decorated graphs approximating this decorated $T$, we can obtain a sequence of finite graphs approximating $G$.
\qed
\medskip
\end{proofof}

In an earlier version of this manuscript \cite{Tarxiv}, we erroneously claimed that we can prove Theorems \ref{sofic} and \ref{main} without the assumption of one-endedness. The argument was based on a unimodular tree-like decomposition of $G$ to pieces that are finite or one-ended, but a mistake was found, so it remains open whether such a decomposition exists for an arbitrary unimodular random graph (or at least planar ones), and whether URPG's with infinitely many ends are sofic.




\subsection{Unimodular random graphs, soficity}\label{definitions}
Given a graph $H$, $x\in V(H)$, $r\in \R^+$, denote by $B(H,x,r)$ the ball of radius $r$ around $x$ in $H$.
Let $\calG_*$ be the set of all locally finite rooted graphs up to rooted isomorphism (isomorphism preserving the root). For a rooted graph $(G,o)$, we denote the respective element (equivalence class) of $\calG_*$ by $[G,o]$, but if there is no ambiguity, we usually just refer to the equivalence class by $(G,o)$ or $G$.
One can make $\calG$ a metric space by defining the distance between two elements $(G,o)$ and $(G',o')$ by $\inf \{2^{-r}: B(G,o,r)\cong B(G',o',r)\}$, where $\cong$ is the relation of being rooted isomorphic. 
A probability measure on $\calG$ is called a random rooted graph. One may also consider {\it marked} rooted graphs, in which case marks (labels) coming from some fixed metric space are also present on some vertices and/or edges. The definitions naturally extend to this setup.
Consider a sequence $G_n$ of finite graphs, and let $o_n$ be a uniformly chosen vertex of $G_n$.  We say that $G_n$ {\it converges} to a random rooted graph $G=(G,o)$ in the {\it local weak} (or {\it Benjamini-Schramm}) sense if for any finite rooted graph $(H,o')$, $\P(B(G_n,o_n,R)\cong (H,o'))\to \P(B(G,o,R)\cong (H,o'))$. In other words, the probability measure corresponding to $(G_n,o_n)$ weakly converges to the probability measure corresponding to $(G,o)$.
If a given random rooted graph $G=(G,o)$ is the Benjamini-Schramm limit of a sequence of finite graphs, we call it {\it sofic}.

Similarly to $\calG_*$, define $\calG_{**}$ as the set of graphs with an ordered pair of vertices, up to isomorphisms preserving the ordered pair. Suppose that $G_0$ is a finite graph, $o_0\in V(G)$ a uniformly chosen root, and $\mu$ the probability measure on $\calG_{*}$ that samples $(G_0,o_0)$. For any Borel $f:\calG_{**}\to [0,\infty]$, we have the equation
\begin{equation}\label{mtp}
\int\sum_{x\in V(G)} f(G,o,x) d\mu([G,o])=\int\sum_{x\in V(G)} f(G,x,o) d\mu([G,o]),
\end{equation}
because both sides are equal to $|V(G_0)|^{-1}\sum_{x,y\in V(G_0)} f(G_0,x,y)$.
Given a probability measure $\mu$ on $\calG_*$, referred to as a {\it random rooted graph}, we
say that it is {\it unimodular} if \eqref{mtp} holds. It is easy to check that sofic graphs are unimodular. 
A major open question is whether the converse is also true.
\begin{question}{(Aldous, Lyons \cite{AL})}
Is every unimodular random (marked) graph sofic?
\end{question}
Whether the Cayley diagram of a group is sofic is a central question itself, and large classes of groups are known to be sofic; see \cite{P} for references or \cite{E2} for some more recent examples. In the class of planar unimodular random graphs that are known to be sofic are unimodular random trees (\cite{B}, \cite{E1}), or Curien's Planar
Stochastic Hyperbolic Triangulations \cite{C}, whose approximability by uniform random triangulations of appropriate genus was shown recently in \cite{BL}.

\section{Unimodular combinatorial planar embeddings for unimodular random planar graphs}\label{sec_embedding}

\subsection{Some tools: Whitney's theorem, generalized Tutte decomposition}

The following theorem was first proved by Whitney for finite simple graphs, and then generalized (to a broader setting than the one below) by Imrich, \cite{I}. 
Recall that having a unique combinatorial embedding up to orientation is the same as having a unique embedding into the plane up to homeomorphisms. 
\begin{theorem}[Whitney, Imrich]\label{Whitney}
Let $G$ be a 3-connected locally finite planar simple graph. Then $G$ has two combinatorial embeddings into the plane, and one arises from the other by inverting all permutations.
\end{theorem}
Now, if one allows (finite bundles of) parallel edges, the theorem remains valid, with the only modification that all the parallel edges between vertices $v$ and $w$ appear in some consecutive order in the $\pi_v$ and $\pi_w$, and the uniqueness of the embedding holds up to arbitrary permutations within these bundles of parallel edges. In what follows, we will apply the theorem in that sense: {\it whenever we take the ``unique" combinatorial embedding of the graph, we mean the random embedding where we first (uniquely) embed the corresponding simple graph, and then add the parallel copies of each edge with a uniform random permutation on them.}

The Tutte decomposition of a finite graph was developed by Tutte in \cite{T} and the uniqueness of the decomposition was shown in \cite{HT}. In \cite{DSS}, Droms, Servatius and Servatius extended the results to infinite locally finite graphs. After preparing the necessary terminology, we will quote their result. More details are found in \cite{DSS}.

Given some graph $G$, a {\it block} of $G$ is a maximal 2-connected subgraph of $G$ (with respect to containment). 
A {\it multilink} is a pair of adjacent vertices together with all of finitely many parallel edges between them. 
A {\it 3-block} is a graph which is a cycle, a finite multilink or a locally finite 3-connected graph, and has at least 3 edges. We mention that by definition both blocks and 3-blocks are allowed to be infinite.

Suppose that $A$ and $B$ are two 
disjoint graphs, and there is a function $f$ that picks one edge $f_A$ from $A$, assigns a tail $f_A^-$ and a head $f_A^+$ to it, and picks an edge $f_B$ from $B$ and assigns a tail $f_B^-$ and a head $f_B^+$ to it.
Then define the {\it edge amalgam of $A$ and $B$ over $f$} as the union of the graphs on $A$ and $B$, with $f_A^-$ identified with $f_B^-$ and $f_A^+$ identified with $f_B^+$, and $f_A$ and $f_B$ removed. See Figure \ref{amalgam1} for an illustration. We denote the edge amalgam of $A$ and $B$ by $A+_f B$ or simply $A+B$. It is easy to check that $A+B$ is 2-connected if and only if $A$ and $B$ are both 2-connected. 
We call a countable labelled tree $T$ an {\it edge amalgam tree} if every vertex $\alpha\in V(T)$ is labelled by a graph $G_\alpha$ (finite or infinite) which are pairwise disjoint, every edge $\{\alpha,\beta\}\in E(T)$ is labelled by a function $f=f(\alpha,\beta)$ that defines an edge amalgam of $G_\alpha$ and $G_\beta$, and finally, every edge $e$ of $G_\alpha$ is amalgamated to at most one other edge, or more precisely, there is at most one function $f$ and $G_\beta$ such that $G_\alpha$ and $G_\beta$ get amalgamated along $e$ (meaning $f_{G_\alpha}=e$). If there is exactly one such $f$, then we call $e$ virtual. One can perform amalgamation over all the edges of $T$ in some order, and the end result will be independent from the order in which the edges of $T$ are chosen, by the last condition. Moreover, an edge will be present in the final graph $\Gamma(T)$ if and only if it is not virtual.

\begin{figure}[h]
\vspace{0.1in}
\begin{center}
\includegraphics[keepaspectratio,scale=0.85]{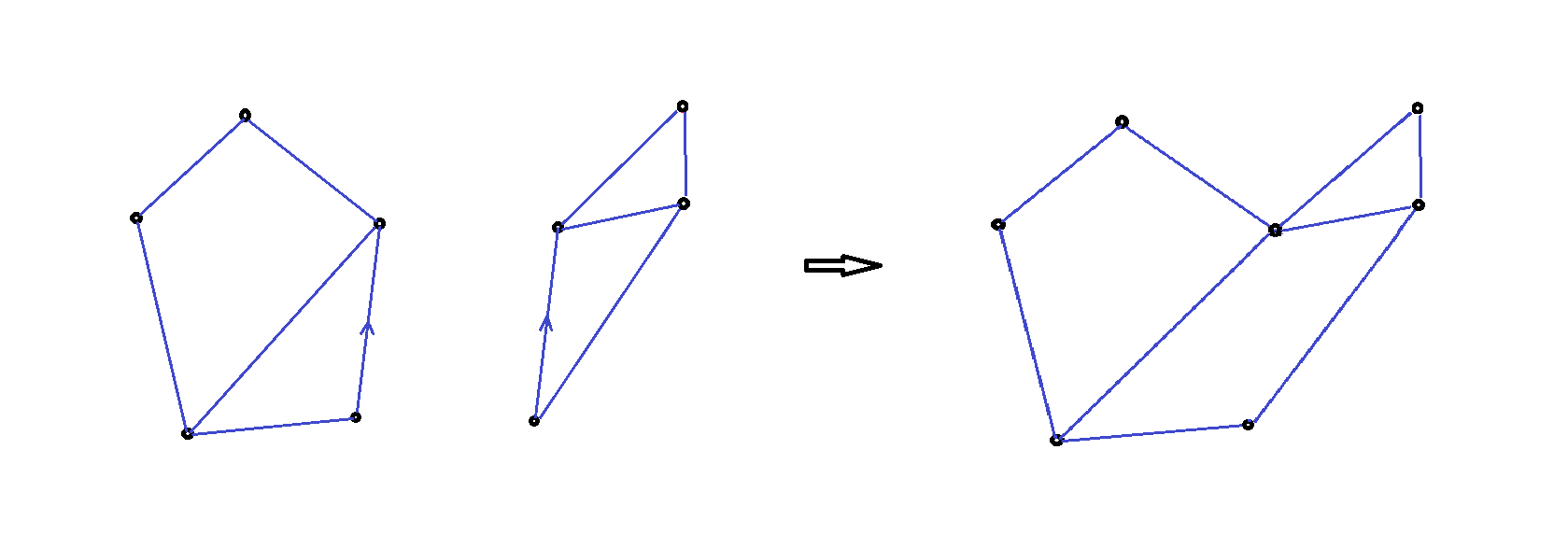}
\caption{Amalgamating two graphs. When both of them come with a combinatorial planar embedding, one obtains a combinatorial embedding of the amalgamated graph.}\label{amalgam1}
\end{center}
\end{figure}

In \cite{DSS} a {\it 3-block tree} is defined, as an edge amalgam tree (where the tree may not be locally finite), where every label $G_\alpha$ is a 3-block, and no two multilinks or two cycles are neighbors in $T$. This latter requirement is only important because it ensures the following uniqueness theorem.  

\begin{theorem}{{\rm (Droms, Servatius, Servatius, \cite{DSS}) }}\label{DSS}
To any locally finite 2-connected graph $G$ there corresponds a unique 3-block tree $T$ such that $\Gamma(T)=G$.
\end{theorem}
The tree $T$ in the theorem can be represented as a labelling of the vertices of $G$, where the labels 
encode which vertices are in the same $G_\alpha$ and
tell the locations of the virtual edges. This can be constructed from $G$ through a local algorithm: from a big enough neighborhood of a vertex we can tell with arbitrary big precision whether it is a cutset of two points, hence we can also determine the labels in a given neighborhood of $o$ with arbitrary precision.
Hence $(G,T)$ is unimodular if $G$ is unimodular.

\subsection{A unimodular combinatorial embedding}

\begin{lemma}\label{3conn}
If $G$ is a 3-connected URPG and it is a simple graph, then the combinatorial embedding in Theorem \ref{Whitney} is unimodular. Consequently, Theorem \ref{embedding} holds for $G$.
\end{lemma}

\begin{proof}
Fix an instance of $(G,o)$; we will show that a large enough neighborhood of $o$ determines $\pi_o$, up to taking the inverse permutations (which we do not mention going forward; we will just flip a coin at the end to decide which of the two to chose).
Let $G_1\leq G_2\leq\ldots$ be a finite exhaustion of $G$, with $o\in G_1$. Fix a neighbor $x$ of $o$ arbitrarily. By the 3-connectedness of $G$, there exist 3 pairwise inner-disjoint paths between $o$ and $x$, and so there is an $N$ such that $G_n$ contains all these paths if $n>N$. In the 3-block tree decomposition of $G_n$, $n>N$, $o$ and $x$ have to belong to the same 3-block (otherwise there is a 2-point separating set between them, more specifically, the endpoints of a virtual edge of the block that contains one of them). So, since $x$ was arbitrary, if $n$ is large enough, all the neighbors of $o$ are in the same 3-block $B_n$ as $o$ in the 3-block tree decomposition of $G_n$, and this 3-block is 3-connected (and not a cycle or multilink, by construction). 
The $B_n$ has a unique embedding by Theorem \ref{Whitney}. Although $B_n$ is not a subgraph of $G_n$, it is easy to see that it is a topological subgraph of $G_n$ (that is, one can replace all the virtual edges of $B_n$ by pairwise inner-disjoint paths in $G_n$). Then it is also a topological subgraph of $G$, because $G_n\leq G$. Therefore the topological subgraph of $G$ that is isomorphic to $B_n$ is also uniquely embedded in the plane.
This implies that the permutation on the neighbors of $o$ defined by the unique embedding of $G$ is the same as the one defined by the embedding of $G_n$. This holds for any $n>N$ proving our claim.
\medskip
\end{proof}

\begin{lemma}\label{edge}
Let $T$ be a 3-block tree consisting of a single edge $\{\alpha,\beta\}$. Then for any combinatorial embedding $\pi^{G_\alpha}$ of $G_\alpha$ and $\pi^{G_\beta}$ of $\G_\beta$ into the plane, there is a combinatorial embedding of $G_\alpha+G_\beta$ in the plane whose restriction to $G_\alpha$ is $\pi^{G_\alpha}$, and restriction to $G_\beta$ is $\pi^{G_\beta}$.
Here uniqueness is understood modulo permutations within bundles of parallel edges.

\noindent
More generally, if $T$ is a finite 3-block tree and a combinatorial embedding $\pi^{G_\alpha}$ is given for every $G_\alpha$, then $\Gamma(T)$ has a unique combinatorial embedding in the plane whose restriction to $G_\alpha$ is $\pi^{G_\alpha}$ for every $\alpha\in V(T)$. One can obtain this embedding by repeatedly applying the previous paragraph to the edges of $T$, in an arbitrary order.
\end{lemma}

\begin{proof}
Denote $A=G_\alpha$ and $B=G_\beta$, and let $\pi^A$ and $\pi^B$ be the respective combinatorial embeddings. Let $f_A$, $f_B$, $f_A^-$, $f_A^+$, $f_B^-$ and $f_B^+$ be as in the definition of the edge-amalgam tree. Let $v^-$ be the vertex that $f_A^-$ and $f_B^-$ is merged into, after the amalgamation. By symmetry, it suffices to define $\pi^{A+B}_{v^-}$.
Suppose that the cyclic permutation $\pi^A_{f_A^-}$ is $(f_A, e_1,\ldots, e_k)$, and the cyclic permutation $\pi^B_{f_B^-}$ is $(f_B, f_1,\ldots, f_\ell)$. 
Then define $\pi^{A+B}_{v^-}$ as the permutation $(e_1,\ldots, e_k, f_1,\ldots, f_\ell)$. 

The second part follows by induction.
\qed
\medskip
\end{proof}

\begin{proofof}{Theorem \ref{embedding}}
First suppose that $G$ is 2-connected.

Consider the unique 3-block tree $T$ that corresponds to $G$, as in Theorem \ref{DSS}. 
Given some $T'\subset T$, denote by ${\rm virt}(T')$ the set of virtual edges in $\cup_{\alpha\in V(T')}E(G_\alpha)$.
For every $\alpha\in V(T)$ fix a combinatorial embedding $\pi^{G_\alpha}$ 
of $G_\alpha$ in the sphere, as follows. If $G_\alpha$ is a multilink, take a uniform cyclic permutation of its edges; if $G_\alpha$ is a 3-connected graph, take uniformly one of the two combinatorial embeddings (as in Theorem \ref{Whitney}); and if $G_\alpha$ is a cycle, take its unique combinatorial embedding.
Then, by Lemma \ref{edge}, for any finite subtree $T'\subset T$, there exists an embedding of the graph $\Gamma(T')
$ in the sphere such that the orientation that this embedding generates when restricted to $G_\alpha$ is 
$\pi^{G_\alpha}$, for every $\alpha\in V(T')$.
Taking an exhaustion of $T$ by finite subtrees $T'$, this gives rise to an embedding of $G\cup {\rm virt}(T)$ with similar properties. For every vertex of $G$ the limiting permutation for the combinatorial embedding is reached in a finite number of steps, hence the limit does not depend on the particular exhaustion taken.
The resulting combinatorial embedding is unimodular: the only non-deterministic part is the embedding of the 3-blocks at the beginning. When the 3-block is finite, its embedding is trivially unimodular, and when it is infinite 3-connected, then apply Lemma \ref{3conn}.

Suppose now that $G$ is an arbitrary connected graph. For each 2-block $C$ pick a unimodular random combinatorial embedding $\pi^C$. Such an embedding exists, as we have just seen.
Consider an arbitrary cutvertex $v\in V(G)$. Let $C_1,\ldots,C_k$ be a listing of the 2-connected components of $G$ that contain $v$.
Denote by $N_i$ the edges of $C_i$ that are incident to $v$. 

For each $i\in\{1,\ldots,k\}$, pick an $e^i_1\in N_i$ uniformly and independently, and let $e^i_1,e^i_2,\ldots, e^i_{|N_i|}$ be the listing 
of the elements of $N_i$ in the order given by $\pi^{C_i}$. Take a uniform cyclic permutation $\delta$ of $\{1,\ldots,k\}$. Define the cyclic ordering $e^1_1,e^1_2,\ldots, e^1_{|N_1|},e^{\delta(1)}_1,e^{\delta(1)}_2,\ldots, e^{\delta(1)}_{|N_{\delta(1)}|},\ldots, e^{\delta^{(k-1)}(1)}_1,e^{\delta^{(k-1)}(1)}_2,\ldots, e^{\delta^{(k-1)}(1)}_{|N_{\delta^{(k-1)}(1)}|}$ on the edges incident to $v$, and call this ordering $\sigma_v$. Because of the tree-like structure that cutvertices define on a graph $G$, the permutations $(\sigma_v)_{v\in V(G)}$ define a combinatorial embedding of $G$ into the plane. For similar reasons as in the 2-connected case, the resulting combinatorial embedding is unimodular.
\qed
\medskip
\end{proofof}


\ \\
\ \\
\noindent
{\'Ad\'am Tim\'ar}\\
University of Iceland
and\\
Alfr\'ed R\'enyi Institute of Mathematics\\
\texttt{madaramit[at]gmail.com}

\end{document}